\documentclass[11pt,reqno]{amsart}
\usepackage{amssymb,latexsym,array,comment}
\usepackage{graphicx}
\usepackage{amsmath}
\usepackage{stmaryrd}
\usepackage{color}
\usepackage{pgf,tikz}
\usepackage{multirow}
\usepackage{setspace}
\usetikzlibrary{arrows}
\usepackage{rotating}
\newcommand\scalemath[2]{\scalebox{#1}{\mbox{\ensuremath{\displaystyle #2}}}}

\newcommand{\N}{{\mathbb N}}
\newcommand{\Z}{{\mathbb Z}}

\newcommand{\Scarf}{\text{Scarf}}
\newcommand{\ord}{\text{ord}}

\newtheorem{theorem}{Theorem}[section]
\newtheorem{example}{Example}[section]
\newtheorem{lemma}{Lemma}[section]
\newtheorem{definition}{Definition}[section]
\newtheorem{prop}{Proposition}[section]

\newtheorem{remark}{Remark}[section]

\definecolor{ffffff}{rgb}{1,1,1}
\definecolor{wwwwww}{rgb}{0.4,0.4,0.4} 

\definecolor{wwwwqq}{rgb}{0.4,0.4,0}
\definecolor{zzqqtt}{rgb}{0.6,0,0.2}
\definecolor{qqqqff}{rgb}{0,0,1}
\definecolor{uququq}{rgb}{0.25,0.25,0.25}
\definecolor{cccccc}{rgb}{0.8,0.8,0.8}
\definecolor{zzzzzz}{rgb}{0.6,0.6,0.6}
\definecolor{wwwwww}{rgb}{0.4,0.4,0.4}
\definecolor{tttttt}{rgb}{0.2,0.2,0.2}

\definecolor{zzttqq}{rgb}{0.27,0.27,0.27}
\definecolor{qqwwcc}{rgb}{0.4,0.4,0.4}
\definecolor{fffftt}{rgb}{0.73,0.73,0.73}
\definecolor{qqzzqq}{rgb}{0.2,0.2,0.2}
\definecolor{ffqqqq}{rgb}{0.33,0.33,0.33}

\doublespace
\begin{document}
\title{A Preliminary Report on Scarf Complexes of Posets}

\begin{abstract}The Scarf complex for lattices is well understood and utilized.  In 2014, the author expanded the use of the term Scarf complex to encompass an infinite set in $\Z^n$ that was generated via an action of a lattice $\Lambda\in\Z^n$ upon a finite subset $A\subset\N^n$.  This paper aims to further generalize the use of Scarf complex by removing the integer lattice completely while maintaining the essence of the Scarf complex.
\end{abstract}
\maketitle
\section{Introduction}

Throughout the literature (e.g. \cite{MS},\cite{PS}), the Scarf complex of a lattice is utilized for its ability to record algebraic data in a combinatorial structure.  In particular, if one has an abelian subgroup $\Lambda$ of $\Z^n$ that intersects the origin only at 0, then under mild conditions, the Scarf complex of $\Lambda$ encodes all the information needed for the minimal free resolution of the lattice ideal $I_{\Lambda}=\{X^{\lambda^+}-X^{\lambda^-} | \lambda\in\Lambda\}\subset k[x_1, \dots, x_n]$.  

The Scarf complex is defined in several equivalent ways, and for the purposes of a combinatorial description, it can be obtained from the concept of neighbors in $\Lambda$.  Two elements $\lambda_1,\lambda_2\in\Lambda$ are neighbors if $(\lambda_1\vee\lambda_2)-\N_{>0}^n\cap\Lambda=\emptyset$.  A subset $A\subseteq\Lambda$ is neighborly if it is pairwise neighborly.  Neighborliness is closed under taking subsets, and $\Scarf(\Lambda)=\{A\subset\Lambda | A \text{ is neighborly}\}/\Lambda$.  We mod out by $\Lambda$ to create a finite (abstract) simplicial complex.

In this paper, we will strip away the excessive structure of $\Z^n$, and concentrate on the most general definition of neighbors possible.  The result will be a Scarf complex that is defined for posets.  Many examples will be given with various posets.

\section{Posets}

Let $(U,\preceq)$ be any poset.  For any element $u\in U$, we will associate the following objects:
\begin{itemize}
\item $\ord(u)$ will denote the order ideal of $u$.  That is, $\ord(u)=\{v\in U | v\preceq u\}$.
\item $\partial \ord(u)=\{v\in \ord(u) | v \text{ has a cover not in } \ord(u)\}\cup \{u\}$.
\item $\ord(u)^{\circ}=\ord(u)-\partial \ord(u)$
\end{itemize}

\begin{definition}Let $(U,\preceq)$ be a poset and let $A\subseteq U$.  If $ B \subseteq A$, we say that $B$ is $A$-neighborly (or just neighborly if there is no confusion) if there exists a $u\in U$ such that $B\subseteq\partial \ord(u)$ and $\ord(u)^{\circ}\cap A=\emptyset$.
\end{definition}

\begin{lemma}\label{closed}
Let $(U,\preceq)$ be a poset and let $A\subseteq U$.  If $B\subset A$ is neighborly, then $B'\subset B$ is neighborly.
\end{lemma}

\begin{proof}
We have that $B$ is neighborly, so there exists a $u\in U$ such that $B\subseteq\partial\ord(u)$ and $\ord(u)^{\circ}\cap A=\emptyset$.  Since $B'\subset B$, we also have that $B'\subseteq\partial\ord(u)$, and hence $B'$ is also neighborly.  
\end{proof}

\begin{remark}
\
\begin{enumerate}
\item If we did not limit $A$-neighborliness to elements of $A$, the result would be copious extraneous neighborly pairs.
\item In Lemma \ref{closed}, we used the same $u\in U$ to show that neighborliness is closed under taking subsets.  In general posets, it is not necessary for there to be only one $u$ to show neighborliness.  This is not always the case, though. In particular, for $\alpha_1,\alpha_2\in \Z^n$, if $\alpha_1-\alpha_2$ has no zero components, then the only possible $u\in\Z^n$ that could be used for neighborliness would be $\alpha_1\vee\alpha_2$.
\end{enumerate}
\end{remark} 

Having shown that neighborliness is closed under taking subsets, we have the following definition.

\begin{definition} Let $(U,\preceq)$ be a poset and let $A\subseteq U$.  The Scarf complex of $U$ with respect to $A$ is $$\Scarf(U,A)=\{B\subseteq A | B \text{ is } A-\text{neighborly}\}$$
\end{definition}

If $U=\Z^n$ (or $\N^n$), then we have the specific case studied in \cite{Mc}, \cite{MS}, and \cite{PS}.  We are now in a more general case, though, and we can proceed to study much more general objects.

\section{Examples}
As motivation for this general Scarf complex, we will examine several standard posets and look at the types of complexes we can get from them.
\subsection{Subsets of $\N^n$}\
\\

Because the Scarf complex has its roots in $\N^n$, it makes sense to look at some other cases in that setting.

\begin{example}\label{divisibility}
Let $U=\N$ and for $u,v\in U$, say $u\preceq v$ if $u$ divides $v$.  Then $(U,\preceq)$ forms a poset.

\begin{lemma}
If $(U,\preceq)$ is the divisibility poset with $U=\N$, then for $u\in U$, the covering set of $u$ is $\{up\hspace{.1cm}|\hspace{.1cm}  p \text{ prime }\}$.
\end{lemma}

\begin{proof}
Clearly, $u$ divides $up$ for all $p$, so $u\preceq up$.  Now suppose $u$ divides $m$ and $m$ divides $up$.  Then there exists $q,r$ such that $m=uq$ and $up=mr$.  Therefore, $uqr=up$ and hence $qr=p$.  Thus, either $r=1$ and $q=p$, or $r=p$ and $q=1$.  If $r=1$, then $m=up$, and if $r=p$, then $m=u$.  Therefore, $up$ covers $u$ for all primes $p$.
\end{proof}

Suppose $A\subset U$ is finite, such that $A=\{a_1, \dots, a_r\}$.  If $a=\prod_1^ra_i$, then $a_i\preceq a$ for all $i$.  Furthermore, $\ord(a)=\{\prod_{i\in I}a_i | I\subseteq [r]\}$ and because there are infinitely many primes, we have infinitely many covers for all elements.  Therefore, $A\subset\partial C_a$, and hence $\ord(a)^{\circ}\cap A=\emptyset$.  Therefore, all (nonempty) subsets of $A$ are neighborly, and hence $\Scarf(U,A)=2^A\setminus\{\}$. Simplicially, this would be the $n$-simplex.
\end{example}

We can also create the opposite effect from Example \ref{divisibility} by creating a poset that has no covers and will result in a Scarf complex comprised entirely of singletons.

\begin{example}\label{singletons}
Let $U$ be the set of prime numbers, and equip $U$ with the divisibility relation, $\preceq$.  Then $(U,\preceq)$ has no comparable elements.  Let $A\subset U$ be a finite subset.  Then for all $a\in A$, we have that $\ord(a)=\{a\}$, $\partial\ord(a)=\{a\}$ and hence $\ord(a)^{\circ}=\emptyset$.  Therefore, all the singletons of $A$ are neighborly, but no pairs are, because $|\ord(a)|=1$ for all $a$.  Thus, $\Scarf(U,A)\cong A$, and the simplicial representation is $|A|$ disjoint points.
\end{example}

\begin{example}\label{powerset}
Let $U=2^{\N}$ and equip $U$ with an ordering defined by inclusion.  That is, if $a,b\in U$, we say $a\preceq b$ if $a\subseteq b$.

For any $A\subseteq U$ and any finite $B\subseteq A$, if we let $B=\{b_1, \dots, b_r\}$, then $\partial\ord(\cup b_i)=2^{\cup b_i}\supset B$.  Furthermore, every finite element of $U$ has infinitely many covers, so $\partial\ord(\cup b_i)=\ord(\cup b_i)$, and hence $\ord(\cup b_i)^{\circ}=\emptyset$.  Therefore, every nonempty finite subset of $A$ is neighborly, and thus $\Scarf{U,A}=2^A\setminus\emptyset$.
\end{example}

\begin{remark}
We can replace $\N$ with any infinite set.  If we consider power sets of finite sets, then we have to be a little careful about the size of the set we are considering neighbors with respect to.
\end{remark}

\subsection{Directed Graphs}\
\\
A type of posets that gives varied Scarf complexes is that which is associated to the collection of acyclic directed graphs.  In the following examples, we will be examining some directed graphs, and in particular, we will look at directed bipartite graphs.  These graphs are exactly the Hasse diagrams of the associated posets, but because we are thinking of directed graphs in these examples, that interpretation will be unnecessary.

Let $G$ be a directed graph with no directed cycles.  If $v,w$ are vertices of $G$, we say $v\preceq w$ if $w$ can be reached via a directed path from $v$.  Since every vertex can be reached from itself (the path of length 0), we have that $\preceq$ is reflexive; since $G$ is acyclic, we have that $\preceq$ is antisymmetric; and by construction, $\preceq$ is transitive.  Hence, we have that $(G,\preceq)$ is a poset.  Henceforth, when we are given a directed graph, we will automatically consider it as a poset on its vertices.

\begin{example}\label{cycle}
Let $G$ be the given graph.

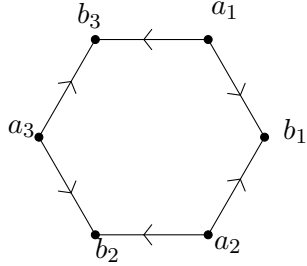
\begin{center}
\begin{tikzpicture}[line cap=round,line join=round,>=triangle 45,x=1.0cm,y=1.0cm]
\draw (0.5400000000000001,5.1580762113533165)-- (1.2900000000000003,3.8590381056766585);
\draw (0.9675,4.417624491117621) -- (0.7980865704891005,4.441057158514988);
\draw (0.9675,4.417624491117621) -- (1.0319134295108996,4.576057158514987);
\draw (0.54,2.56)-- (-0.96,2.56);
\draw (-0.315,2.56) -- (-0.21,2.695);
\draw (-0.315,2.56) -- (-0.21,2.4250000000000003);
\draw (-1.7100000000000004,3.8590381056766594)-- (-0.9599999999999999,5.158076211353317);
\draw (-1.2825,4.599489825912355) -- (-1.2180865704891006,4.4410571585149885);
\draw (-1.2825,4.599489825912355) -- (-1.4519134295108995,4.576057158514988);
\draw (-1.7100000000000004,3.8590381056766594)-- (-0.96,2.56);
\draw (-1.2825,3.1185863854409632) -- (-1.4519134295108995,3.1420190528383296);
\draw (-1.2825,3.1185863854409632) -- (-1.2180865704891006,3.277019052838329);
\draw (0.54,2.56)-- (1.2900000000000003,3.8590381056766585);
\draw (0.9675,3.3004517202356953) -- (1.0319134295108996,3.1420190528383296);
\draw (0.9675,3.3004517202356953) -- (0.798086570489101,3.277019052838329);
\draw (0.5400000000000001,5.1580762113533165)-- (-0.9599999999999999,5.158076211353317);
\draw (-0.315,5.158076211353317) -- (-0.21,5.293076211353316);
\draw (-0.315,5.158076211353317) -- (-0.21,5.023076211353317);
\draw (0.44,5.8) node[anchor=north west] {$a_1$};
\draw (1.4000000000000001,4.24) node[anchor=north west] {$b_1$};
\draw (0.48,2.68) node[anchor=north west] {$a_2$};
\draw (-1.1,2.66) node[anchor=north west] {$b_2$};
\draw (-2.26,4.22) node[anchor=north west] {$a_3$};
\draw (-1.34,5.78) node[anchor=north west] {$b_3$};
\begin{scriptsize}
\draw [fill=black] (-0.96,2.56) circle (1.5pt);
\draw [fill=black] (0.54,2.56) circle (1.5pt);
\draw [fill=black] (1.2900000000000003,3.8590381056766585) circle (1.5pt);
\draw [fill=black] (0.5400000000000001,5.1580762113533165) circle (1.5pt);
\draw [fill=black] (-0.9599999999999999,5.158076211353317) circle (1.5pt);
\draw [fill=black] (-1.7100000000000004,3.8590381056766594) circle (1.5pt);
\end{scriptsize}
\end{tikzpicture}
\end{center}

Then we have the following covering relations:

\begin{center}$\begin{array}{ccc} a_1\preceq b_1 & a_2\preceq b_1 & a_3\preceq b_2 \\
									a_1\preceq b_3 & a_3\preceq b_2 & a_3\preceq b_3
\end{array}$

\end{center}

Using these relations, we compute the order ideals, boundaries and interiors of each element.

\begin{center}
$\begin{array}{llll}
\ord(a_i) = \{a_i\} & \ord(b_1) = \{a_1,a_2,b_1\} & \ord(b_2)=\{a_2,a_3,b_2\} & \ord(b_3)=\{a_1,a_3,b_3\}\\
\partial \ord(a_i) = \{a_i\} & \partial\ord(b_1) = \{a_1,a_2, b_1\} & \partial\ord(b_2)=\{a_2,a_3, b_2\} & \partial\ord(b_3)=\{a_1,a_3, b_3\}\\
\ord(a_i)^{\circ} = \emptyset & \ord(b_1)^{\circ} = \emptyset & \ord(b_2)^{\circ}=\emptyset & \ord(b_3)^{\circ}=\emptyset
\end{array}$
\end{center}

Then we have the following Scarf complexes associated to $G$.

\begin{center}\begin{tabular}{| c | c | c |}
\hline
$A$ & $\Scarf(G,A)$ & simplicial representation\\
\hline
\raisebox{1cm}{$\{a_1,a_2,a_3\}$} & \raisebox{1cm}{\parbox{5cm}{\centering{$\{\{a_1\},\{a_2\},\{a_3\}$,\\ $\{a_1,a_2\},\{a_2,a_3\},\{a_1,a_3\}\}$}}} & \scalemath{.8}{\begin{tikzpicture}[line cap=round,line join=round,>=triangle 45,x=1.0cm,y=1.0cm]
\draw (0.0,-0.0)-- (1.0,2.0);
\draw (1.0,2.0)-- (2.0,0.0);
\draw (2.0,0.0)-- (0.0,0.0);
\begin{scriptsize}
\draw [fill=black] (0.0,0.0) circle (1.5pt);
\draw[color=black] (-0.2,0.0) node {$a_1$};
\draw [fill=black] (1.0,2.0) circle (1.5pt);
\draw[color=black] (1,2.2) node {$a_2$};
\draw [fill=black] (2.0,0.0) circle (1.5pt);
\draw[color=black] (2.3,0.0) node {$a_3$};
\end{scriptsize}
\end{tikzpicture}}\\
\hline
\raisebox{.15cm}{$\{a_1,a_2,b_1\}$} & \raisebox{.15cm}{$\{\{a_1\},\{a_2\},\{b_2\},\{a_2,b_2\}\}$} & \begin{tikzpicture}[line cap=round,line join=round,>=triangle 45,x=1.0cm,y=1.0cm]
\begin{scriptsize}
\draw (1.0,0) -- (2.0,0);
\draw [fill=black] (0.0,0.0) circle (1.5pt);
\draw[color=black] (-.2,0.0) node {$a_1$};
\draw [fill=black] (1.0,0) circle (1.5pt);
\draw [color=black] (1.0,-.2) node {$a_2$};
\draw [fill=black] (2.0,0.0) circle (1.5pt);
\draw[color=black] (2.3,0.0) node {$b_2$};
\end{scriptsize}
\end{tikzpicture}\\
\hline
\raisebox{1cm}{$\{a_1,a_2,a_3,b_1,b_2,b_3\}$} & \raisebox{1cm}{\parbox{5cm}{\centering{$\{\{a_1\},\{a_2\},\{a_3\}$,\\ $\{a_1,a_2\},\{a_2,a_3\},\{a_1,a_3\}\}$}}} & \scalemath{.8}{\begin{tikzpicture}[line cap=round,line join=round,>=triangle 45,x=1.0cm,y=1.0cm]
\begin{scriptsize}
\draw [fill=black] (0.0,0.0) circle (1.5pt);
\draw[color=black] (-0.2,0.0) node {$a_1$};
\draw [fill=black] (1.0,2.0) circle (1.5pt);
\draw[color=black] (1,2.2) node {$a_2$};
\draw [fill=black] (2.0,0.0) circle (1.5pt);
\draw[color=black] (2.3,0.0) node {$a_3$};
\end{scriptsize}
\end{tikzpicture}}\\
\hline
\end{tabular}
\end{center}
\end{example}

\begin{remark}
In this example, we opted to view $G$ as a cycle directed in such a way that the vertices alternated as sinks and sources.  We could just as well have viewed $G$ as a directed bipartite graph.  As we mentioned earlier, this has a natural relation to the Hasse diagram of $G$ as a poset.  The desire to consider this representation of $G$ will be evident in the following proposition.
\end{remark}

\begin{prop}\label{prop1}
Let $G$ be a cycle on $2n$ vertices labeled $a_1, b_1, a_2, b_2, \cdots, b_n, a_n$ such that each $a_i$ is a source and each $b_i$ is a sink.  If $A=\{a_i | i=1, \dots, n\}$, then $\Scarf(G,A)$ has a simplicial representation isomorphic to a cycle on $n$ vertices.
\end{prop}

\begin{proof}
We can compute it directly, using Example \ref{cycle} as a guide. We first note that for each $a_i$, $\{a_i\}\subset\partial \ord(a_i)=\{a_i\}$, and also that $\ord(a_i)^{\circ}=\emptyset$ for all $a_i$.  Therefore, each $\{a_i\}$ is neighborly.  It remains to show that $\{a_i,a_{i+1}\}$ is a neighborly set for $i=1, \dots, n$ with $a_{n+1}$ identified with $a_1$.  To see this, notice that $\ord(b_i)\supset\{a_i,a_{i+1}\}$ and that $\partial\ord(b_i)=\ord(b_i)$.  Since $\ord(b_i)\cap A=\emptyset$ for all $i$, we conclude that $\{a_i,a_{i+1}\}$ is a neighborly set as required.  Finally, by construction, there are no other vertices that have boundary sets that contain any non-sequential pairs of elements, and no boundary sets of cardinality greater than 2.  Thus, we have exhausted all possible neighborly sets, and have $\Scarf(G,A)=\{\{a_i\},\{a_i,a_{i+1}\}\}$, which can be represented simplicially by the cycle on $n$ vertices.

\end{proof}

\begin{example}\label{bipartite}
Let $G$ be a complete bipartite graph with bipartitions $A$ and $B$ with $|A|=m$ and $|B|=n\geq 2$ directed from $A$ to $B$. That is, $G=k_{m,n}$ with directions. Note that this is identical to a Hasse diagram with two levels such that every element of the second level covers every element of the first level.  Order $G$ by reachability to make a poset.

If $A=\{a_1, \dots, a_m\}, B=\{b_1, \dots, b_n\}$, then $a_i\preceq b_j$ for all $i,j$.  Furthermore, $\ord(b_i)=A\cup\{b_i\}$ and we have that $\partial \ord(b_i)=\ord(b_i)$ because each $a_j$ is covered by $b_{i+1}$ or $b_{i-1}$.  Therefore, $\ord(b_i)^{\circ}=\emptyset$, and hence $B\cap \ord(b_i)^{\circ}=\emptyset$.  Therefore, all subsets of $B$ are $A$-neighborly, including $B$ itself.  Therefore, $\Scarf(G,B)=2^{B}\setminus\{\}$ and has a simplicial representation that is the $|B|$-simplex.
\end{example}

Example \ref{bipartite} leads us to an impressive theorem concerning Scarf complexes.

\begin{theorem}\label{anycomplex}
If $X$ is any connected finite simplicial complex, then there is a directed graph $G$ and a subset of its vertices $A$ such that the simplicial representation of $\Scarf(G,A)$ is $X$.
\end{theorem}

\begin{proof}
We can give the proof via an algorithm.
\begin{enumerate}
\item Label the vertices of $X$.
\item Decompose $X$ into maximal subsimplices. By connectedness, these will all be 2-simplices or larger.
\item By Example \ref{bipartite}, each of these simplices is realized by a $k_{m,n}$ for $n\geq 2$.
\item Label the $k_{m,n}$'s with the labels of $X$ on the source side, and $x_i$ on the sink side, repeating labels of $X$ as needed, but not repeating any $x_i$.
\item Glue the $k_{m,n}$'s together by the labels to make a directed bipartite graph $G$.
\end{enumerate}

If $A=\{\text{sources}\}$, then $\Scarf(G,A)\cong X$.
\end{proof}

\begin{example}\label{theorem}
Let $X$ be the following simplicial complex composed of a 3-simplex, three 2-simplices and five 1-simplices.

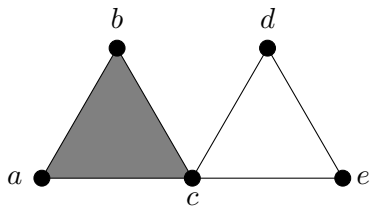
\begin{center}
\begin{tikzpicture}[line cap=round,line join=round,>=triangle 45,x=1.0cm,y=1.0cm]
\fill[fill=gray](0.0,-0.0) -- (2.0,0.0) -- (1.0,1.73) -- cycle;
\draw (0.0,-0.0)-- (2.0,0.0);
\draw (2.0,0.0)-- (1.0,1.73);
\draw (1.0,1.73)-- (0.0,-0.0);
\draw (2.0,0.0)-- (4.0,0.0);
\draw (4.0,0.0)-- (3.0,1.73);
\draw (3.0,1.73)-- (2.0,0.0);
\draw (1,2.4) node[anchor=north] {$b$};
\draw (-0.6,0.0) node[anchor=west] {$a$};
\draw (2,-0.5) node[anchor=south] {$c$};
\draw (3,2.4) node[anchor=north] {$d$};
\draw (4.5,0.0) node[anchor=east] {$e$};
\begin{scriptsize}
\draw [fill=black] (0.0,-0.0) circle (3.0pt);
\draw [fill=black] (2.0,0.0) circle (3.0pt);
\draw [fill=black] (1.0,1.73) circle (3.0pt);
\draw [fill=black] (4.0,0.0) circle (3.0pt);
\draw [fill=black] (3.0,1.73) circle (3.0pt);
\end{scriptsize}
\end{tikzpicture}
\end{center}

Now we decompose $X$ into maximal subsimplices.

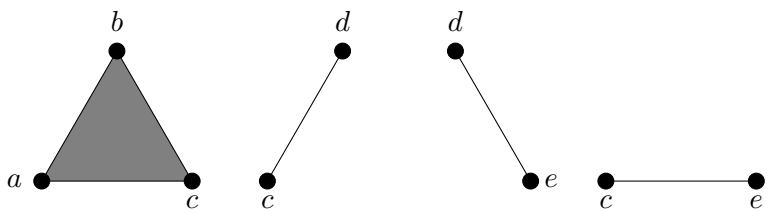
\begin{center}
\begin{tikzpicture}[line cap=round,line join=round,>=triangle 45,x=1.0cm,y=1.0cm]
\fill[fill=gray](0.0,-0.0) -- (2.0,0.0) -- (1.0,1.73) -- cycle;
\draw (0.0,-0.0)-- (2.0,0.0);
\draw (2.0,0.0)-- (1.0,1.73);
\draw (1.0,1.73)-- (0.0,-0.0);

\draw (4.0,1.73)-- (3.0,0.0);
\draw (6.5,0.0)-- (5.5,1.73);
\draw (7.5,-0.0)-- (9.5,0.0);

\draw (1,2.4) node[anchor=north] {$b$};
\draw (-0.6,0.0) node[anchor=west] {$a$};
\draw (2,-0.5) node[anchor=south] {$c$};

\draw (3,-0.5) node[anchor=south] {$c$};
\draw (4,2.4) node[anchor=north] {$d$};

\draw (5.5,2.4) node[anchor=north] {$d$};
\draw (7.0,0.0) node[anchor=east] {$e$};

\draw (7.5,-0.5) node[anchor=south] {$c$};
\draw (9.5,-0.5) node[anchor=south] {$e$};

\begin{scriptsize}
\draw [fill=black] (0.0,-0.0) circle (3.0pt);
\draw [fill=black] (2.0,0.0) circle (3.0pt);
\draw [fill=black] (1.0,1.73) circle (3.0pt);

\draw [fill=black] (3.0,0.0) circle (3.0pt);
\draw [fill=black] (4.0,1.73) circle (3.0pt);

\draw [fill=black] (6.5,0.0) circle (3.0pt);
\draw [fill=black] (5.5,1.73) circle (3.0pt);

\draw [fill=black] (7.5,0.0) circle (3.0pt);
\draw [fill=black] (9.5,0.0) circle (3.0pt);
\end{scriptsize}
\end{tikzpicture}
\end{center}

These give rise to the following four labeled bipartite graphs directed from left to right.

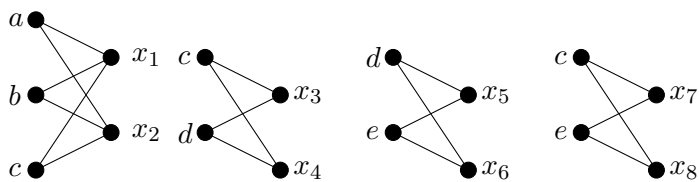
\begin{center}
\begin{tikzpicture}[line cap=round,line join=round,>=triangle 45,x=1.0cm,y=1.0cm]
\draw [fill=black] (0.0,1.0) circle (3.0pt);
\draw [fill=black] (0.0,0.0) circle (3.0pt);
\draw [fill=black] (0.0,-1.0) circle (3.0pt);
\draw [fill=black] (1,0.5) circle (3.0pt);
\draw [fill=black] (1,-0.5) circle (3.0pt);

\draw [fill=black] (2.25,0.5) circle (3.0pt);
\draw [fill=black] (2.25,-0.5) circle (3.0pt);
\draw [fill=black] (3.25,0.0) circle (3.0pt);
\draw [fill=black] (3.25, -1) circle (3.0pt);

\draw [fill=black] (4.75,0.5) circle (3.0pt);
\draw [fill=black] (4.75,-0.5) circle (3.0pt);
\draw [fill=black] (5.75,0.0) circle (3.0pt);
\draw [fill=black] (5.75,-1) circle (3.0pt);

\draw [fill=black] (7.25,0.5) circle (3.0pt);
\draw [fill=black] (7.25,-0.5) circle (3.0pt);
\draw [fill=black] (8.25,0.0) circle (3.0pt);
\draw [fill=black] (8.25,-1) circle (3.0pt);

\draw (4.75,0.5) -- (5.75,0);
\draw (4.75,-.5) -- (5.75,0);
\draw (7.25,0.5) -- (8.25,0);
\draw (7.25,-.5) -- (8.25,0); 

\draw (4.75,0.5) -- (5.75,-1);
\draw (4.75,-.5) -- (5.75,-1);
\draw (7.25,0.5) -- (8.25,-1);
\draw (7.25,-.5) -- (8.25,-1); 

\draw (0,1) -- (1,.5);
\draw (0,0) -- (1,.5);
\draw (0,-1) -- (1,.5);
\draw (0,1) -- (1,-.5);
\draw (0,0) -- (1,-.5);
\draw (0,-1) -- (1,-.5);

\draw (2.25,0.5) -- (3.25,0);
\draw (2.25,-.5) -- (3.25,0);
\draw (2.25,0.5) -- (3.25,-1);
\draw (2.25,-.5) -- (3.25,-1);

\draw (-.5,1) node[anchor=west] {$a$};
\draw (-.5,0) node[anchor=west] {$b$};
\draw (-.5,-1) node[anchor=west] {$c$};
\draw (1.8,0.5) node[anchor=east] {$x_1$};
\draw (1.8,-0.5) node[anchor=east] {$x_2$};

\draw (1.75,.5) node[anchor=west] {$c$};
\draw (1.75,-.5) node[anchor=west] {$d$};
\draw (3.95,0) node[anchor=east] {$x_3$};
\draw (3.95,-1) node[anchor=east] {$x_4$};

\draw (4.25,.5) node[anchor=west] {$d$};
\draw (4.25,-.5) node[anchor=west] {$e$};
\draw (6.45,0) node[anchor=east] {$x_5$};
\draw (6.45,-1) node[anchor=east] {$x_6$};

\draw (6.75,.5) node[anchor=west] {$c$};
\draw (6.75,-.5) node[anchor=west] {$e$};
\draw (8.95,0) node[anchor=east] {$x_7$};
\draw (8.95,-1) node[anchor=east] {$x_8$};
\end{tikzpicture}
\end{center}

To complete the example, we glue the bipartite graphs together along their labels to get $G$.

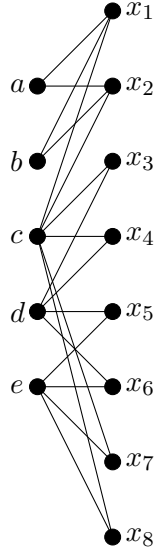
\begin{center}
\begin{tikzpicture}[line cap=round,line join=round,>=triangle 45,x=1.0cm,y=1.0cm]
\draw [fill=black] (0.0,2.0) circle (3.0pt);
\draw [fill=black] (0.0,1.0) circle (3.0pt);
\draw [fill=black] (0.0,0.0) circle (3.0pt);
\draw [fill=black] (0.0,-1.0) circle (3.0pt);
\draw [fill=black] (0.0,-2.0) circle (3.0pt);

\draw [fill=black] (1.0,3.0) circle (3.0pt);
\draw [fill=black] (1.0,2.0) circle (3.0pt);
\draw [fill=black] (1.0,1.0) circle (3.0pt);
\draw [fill=black] (1.0,0.0) circle (3.0pt);
\draw [fill=black] (1.0,-1.0) circle (3.0pt);
\draw [fill=black] (1.0,-2.0) circle (3.0pt);
\draw [fill=black] (1.0,-3.0) circle (3.0pt);
\draw [fill=black] (1.0,-4.0) circle (3.0pt);

\draw (-.5,2) node[anchor=west] {$a$};
\draw (-.5,1) node[anchor=west] {$b$};
\draw (-.5,0) node[anchor=west] {$c$};
\draw (-.5,-1) node[anchor=west] {$d$};
\draw (-.5,-2) node[anchor=west] {$e$};

\draw (1.7,2) node[anchor=east] {$x_2$};
\draw (1.7,1) node[anchor=east] {$x_3$};
\draw (1.7,0) node[anchor=east] {$x_4$};
\draw (1.7,-1) node[anchor=east] {$x_5$};
\draw (1.7,-2) node[anchor=east] {$x_6$};
\draw (1.7,3) node[anchor=east] {$x_1$};
\draw (1.7,-3) node[anchor=east] {$x_7$};
\draw (1.7,-4) node[anchor=east] {$x_8$};

\draw (0,2) -- (1,3);
\draw (0,2) -- (1,2);
\draw (0,1) -- (1,3);
\draw (0,1) -- (1,2);
\draw (0,0) -- (1,3);
\draw (0,0) -- (1,2);
\draw (0,0) -- (1,1);
\draw (0,0) -- (1,0);
\draw (0,0) -- (1,-3);
\draw (0,0) -- (1,-4);
\draw (0,-1) -- (1,1);
\draw (0,-1) -- (1,0);
\draw (0,-1) -- (1,-1);
\draw (0,-1) -- (1,-2);
\draw (0,-2) -- (1,-1);
\draw (0,-2) -- (1,-2);
\draw (0,-2) -- (1,-3);
\draw (0,-2) -- (1,-4);
\end{tikzpicture}
\end{center}

After this construction, $\Scarf(G,\{a,b,c,d,e\})\cong X$.
\end{example}

\subsection{Hasse Diagrams}

If $U$ is a finite poset, then we can create its Hasse diagram and easily read all the covering relations off the diagram.  In fact, we can even read the interior of an order ideal easily from certain diagrams.  In particular, we will first examine the type A positive root poset. The type A positive root poset of size 5, denoted $A_5$ is pictured below.

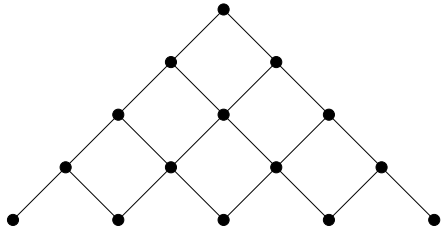
\begin{center}
\scalemath{.7}{\begin{tikzpicture}[line cap=round,line join=round,>=triangle 45,x=1.0cm,y=1.0cm]
\begin{scriptsize}
\draw [fill=black] (5.0,6.0) circle (3pt);
\draw [fill=black] (4.0,5.0) circle (3pt);
\draw [fill=black] (6.0,5.0) circle (3pt);
\draw [fill=black] (3.0,4.0) circle (3pt);
\draw [fill=black] (5.0,4.0) circle (3pt);
\draw [fill=black] (7.0,4.0) circle (3pt);
\draw [fill=black] (2.0,3.0) circle (3pt);
\draw [fill=black] (4.0,3.0) circle (3pt);
\draw [fill=black] (6.0,3.0) circle (3pt);
\draw [fill=black] (8.0,3.0) circle (3pt);
\draw [fill=black] (1.0,2.0) circle (3pt);
\draw [fill=black] (3.0,2.0) circle (3pt);
\draw [fill=black] (5.0,2.0) circle (3pt);
\draw [fill=black] (7.0,2.0) circle (3pt);
\draw [fill=black] (9.0,2.0) circle (3pt);
\end{scriptsize}
\draw (5,6) -- (1,2);
\draw (6,5) -- (3,2);
\draw (7,4) -- (5,2);
\draw (8,3) -- (7,2);
\draw (5,6) -- (9,2);
\draw (4,5) -- (7,2);
\draw (3,4) -- (5,2);
\draw (2,3) -- (3,2);
\end{tikzpicture}}
\end{center}

We will label elements of $A_n$ as $a_{i,j}$ where $i$ is the number row starting from the bottom row and counting up, and $j$ is the number element in the $i^{th}$ row counting left to right.  As with other posets, we have many differently shaped and sized Scarf complexes that we can obtain from $A_n$ by choosing different subsets.

\begin{prop}\label{m-simplex}
If $A=\ord(a_{n-4,3})\subset A_n$, then $\Scarf(A_n,A)\cong (2n-5)$-simplex.
\end{prop}

\begin{proof}
Diagramatically, we have the following setup.

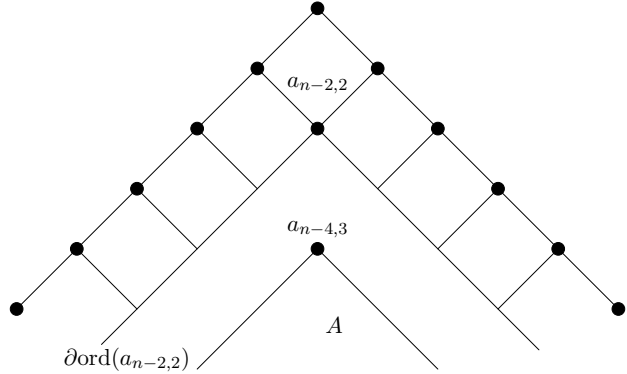
\begin{center}
\scalemath{.8}{\begin{tikzpicture}[line cap=round,line join=round,>=triangle 45,x=1.0cm,y=1.0cm]
\draw (5.0,4.0)-- (1.41,0.41);
\draw (5.0,4.0)-- (8.68,0.32);
\draw (5.0,2.0)-- (3.0,0.0);
\draw (5.0,2.0)-- (7.0,0.0);
\draw (4.36,5) node[anchor=north west] {$a_{n-2,2}$};
\draw (.65,.5) node[anchor=north west] {$\partial\ord(a_{n-2,2})$};
\draw (4.36,2.6) node[anchor=north west] {$a_{n-4,3}$};
\draw (5,1) node[anchor=north west] {$A$};
\begin{scriptsize}
\draw [fill=black] (5.0,6.0) circle (3pt);
\draw [fill=black] (4.0,5.0) circle (3pt);
\draw [fill=black] (6.0,5.0) circle (3pt);
\draw [fill=black] (3.0,4.0) circle (3pt);
\draw [fill=black] (5.0,4.0) circle (3pt);
\draw [fill=black] (7.0,4.0) circle (3pt);
\draw [fill=black] (2.0,3.0) circle (3pt);
\draw [fill=black] (8.0,3.0) circle (3pt);
\draw [fill=black] (1.0,2.0) circle (3pt);
\draw [fill=black] (9.0,2.0) circle (3pt);
\draw [fill=black] (0.0,1.0) circle (3pt);
\draw [fill=black] (10.0,1.0) circle (3pt);
\draw [fill=black] (5.0,2.0) circle (3pt);
\draw (4,5) -- (5,4);
\draw (3,4) -- (4,3);
\draw (2,3) -- (3,2);
\draw (1,2) -- (2,1);
\draw (6,5) -- (5,4);
\draw (7,4) -- (6,3);
\draw (8,3) -- (7,2);
\draw (9,2) -- (8,1);
\draw (5,6) -- (0,1);
\draw (5,6) -- (10,1);
\end{scriptsize}
\end{tikzpicture}}
\end{center}

By construction, $$|\partial\ord(a_{n-2,2})|={{(n-2)+1}\choose{2}}-{{(n-4)+1}\choose{2}},$$ the difference of the $(n-2)^{nd}$ and $(n-4)^{th}$ triangular numbers.  Simplifying, $|\partial\ord(a_{n-2,2})|=2n-5$.  The diagram shows that $\ord(a_{n-2,2})^{\circ}\cap A=\emptyset$, so we have a neighborly set of size $2n-5$, which gives rise to the simplex on $2n-5$ vertices as required.
\end{proof}

We can also choose other subsets of $A_n$ to get interesting classes of Scarf complexes.

\begin{prop}
If $A=\{a_{1,2},a_{1,3},\dots, a_{1,n-1},a_{2,1},a_{2,n-1}\}\subset A_n$, then $\Scarf(A_n,A)$ is the path on $n$ vertices.
\end{prop}

\begin{proof}
By direct computation, $\partial\ord(a_{2,i})=\{a_{1,i},a_{1,i+1},a_{2,i}\}$ , and $\ord(a_{2,i})^{\circ}=\{\}$ for all $i\neq 1,n-2$.  Additionally, $\partial\ord(a_{2,1})=\{a_{1,2},a_{2,1}\}$, and $\ord(a_{2,1})^{\circ}=\{a_{1,1}\}$, as well as $\partial\ord(a_{2,n-1})=\{a_{1,n},a_{2,n-1}\}$, and $\ord(a_{2,n-1})^{\circ}=\{a_{1,n}\}$.  This gives us the path on $n$ vertices labeled by $a_{2,1}, a_{1,2}, \dots, a_{1,n-1},a_{2,n-1}$.

Routine computation shows that any element of $A_n$ having any 3-element subset of $A$ in its boundary necessary has a nonempty intersection with $A$.
\end{proof}

As a final example, we have as Scarf complexes of $A_n$ the collection of disconnected complexes consisting a path of a certain length and two disconnected vertices.

\begin{prop}
If $A=\{a_{i,1}, \dots, a_{i,n-i+1}\}\subset A_n$ for $i\leq n-2$, then $\Scarf(A_n,A)$ is a path on $n-i-1$ vertices together with 2 disjoint vertices.
\end{prop}

\begin{proof}
We begin first with $i<n-2$. For the path, we choose the elements $\{a_{i+1,2},\dots, a_{i+1,n-i-1}\}$ to take the boundaries of, and for the disjoint points, we choose the elements $\{a_{i,1},a_{1,n-i+1}\}$.

If $i=n-2$, then we have a collection of $n-2$ disjoint points.  To get this, choose points $\{a_{n-1,1},a_{n-1,2}\}$.  The remainder is a routine computation.
\end{proof}

\section{Conclusion}

Scarf complexes from posets is a very new idea, and has yet to prove to be anything more than an interesting combinatorial trick.  However, the algebraic fruit borne from $\Scarf(\Z^n,\Lambda)$ when $\Lambda$ is a subgroup of $\Z^n$ intersecting $\N^n$ only at the origin certainly warrants a detailed look at the more generalized structure.

There are several avenues the author is currently pursuing in this work.  The first is to establish what classes of posets yield meaningful collections of Scarf complexes.  Lattices are the most promising posets in this direction.  Having unique suprema between elements is akin to what happens in the $\Lambda\subset\Z^n$ case.  

In another direction, it can be shown that if $U$ and $V$ are disjoint posets with $A\subset U$ and $B\subset V$, then $\Scarf(U+V,A+B)\cong\Scarf(U,A)\cup\Scarf(V,B)$ where $+$ is the direction sum of posets.  Similarly, if $\oplus$ is the ordinal sum of posets, then $\Scarf(U\oplus V,A\oplus B)\cong\Scarf(U,A)\cup\Scarf(V,B)$, as well.  It is not yet known what similar statements can be made about Cartesian products or ordinal products of posets.

It is the hope of the author to find algebraic structures generalizing the resolutions of monomial and binomial ideals that mirror the generalization on the combinatorial structures.

\end{document}